\documentclass[11pt,a4paper]{article}

\usepackage{amsmath,amssymb,amsthm,amsfonts,latexsym,graphicx,subfigure}
\usepackage[all,import]{xy}
\usepackage[numbers,sort&compress]{natbib}
\usepackage{indentfirst}
\usepackage{fancyhdr}
\usepackage{bbm}
\usepackage{hyperref}
\usepackage{enumerate,tikz}
\usepackage{accents,cases}
\usepackage{multirow}
\usepackage[lined, ruled, linesnumbered, longend]{algorithm2e}

\usepackage{array}
\newcommand{\PreserveBackslash}[1]{\let\temp=\\#1\let\\=\temp}
\newcolumntype{C}[1]{>{\PreserveBackslash\centering}p{#1}}
\newcolumntype{R}[1]{>{\PreserveBackslash\raggedleft}p{#1}}
\newcolumntype{L}[1]{>{\PreserveBackslash\raggedright}p{#1}}

\DeclareMathOperator*{\Sgn}{\ensuremath{Sgn}}

\DeclareMathOperator*{\sgn}{\ensuremath{sgn}}

\DeclareMathOperator*{\vol}{\ensuremath{vol}}
 
 \newcommand{\Span}{\mathrm{span}}
 \newcommand{\spec}{\mathrm{spec}}
 \newcommand{\Spec}{\textbf{Spec}}

\makeatletter
\def\wbar{\accentset{{\cc@style\underline{\mskip8mu}}}}

\makeatother

\renewcommand{\vec}[1]{\mbox{\boldmath \small $#1$}}

\numberwithin{equation}{section}

\theoremstyle{plain}
\newtheorem{theorem}{Theorem}

\newtheorem{defn}{Definition}

\newtheorem{remark}{Remark}

\newtheorem{cor}{Corollary}
\newtheorem{pro}{Proposition}
\newtheorem{example}{Example}

\begin{document}
\bibliographystyle{unsrt}
\title{Spectrum of signless 1-Laplacian on simplicial complexes}
\author{Xin Luo\footnotemark[1]~\footnotemark[2],
\and Dong Zhang\footnotemark[3]}
\renewcommand{\thefootnote}{\fnsymbol{footnote}}
\footnotetext[1]{College of Mathematics and Economics, Hunan University, Changsha 410082, P.R. China.

Email address: {\tt xinlnew@163.com} (Xin Luo).
}

\footnotetext[2]{Current Address: Academy of Mathematics and Systems Science, Chinese Academy of Sciences, Beijing 100190, P.R. China.
}

\footnotetext[3]{LMAM and School of Mathematical Sciences, Peking University, Beijing 100871, P.R. China.\\
Email addresses:
{\tt dongzhangmath@gmail.com} \; and \; {\tt 13699289001@163.com} (Dong Zhang).
}
\date{}
\maketitle

\begin{abstract}

We introduce the signless 1-Laplacians and the dual Cheeger constants on simplicial complexes.   
The connection of its spectrum to the combinatorial properties like independence number,  chromatic number and dual Cheeger constant is investigated. Our estimates  can be comparable to  Hoffman's bounds in virtue of Laplacian on simplicial complexes. An interesting inequality relating multiplicity of the largest eigenvalue, independence number and chromatic number are provided, which could be regarded as a variant version of Lovasz sandwich theorem. Also, the behavior of the operator under the topological operations of wedge and duplication of motifs is studied.
The Courant nodal domain theorem in spectral theory is extended to the case of signless 1-Laplacian on complexes.  
\vspace{0.3cm}

\end{abstract}

\tableofcontents

\section{Introduction}


It is widely known that the discrete Laplacian on graphs has been studied for a long history, 
and B. Eckmann has generalized it to simplicial complexes which encodes the topological and combinatorial  data \cite{eck}.
Some recent remarkable results on discrete Laplacian for simplicial complex involve lots of algebraic and geometric aspects for complex \cite{HorakJost2012,HorakJost2013,BGP19,Golubev,PR17}, which are applied in the study of expansion, mixing, colorings, random walks, geometric and topological overlap. Other relevant operators like adjacency operator and Hecke operator (see \cite{EGL15}), as well as signless Laplacians \cite{KO17}, are systematically studied with many interesting applications \cite{KL14,Lubotzky14,Lubotzky17,LP20}.

On the other hand, the (signless) 1-Laplacian on graph has been also investigated
systematically, see \cite{HeinBuhler2010,Chang2015,ChangShaoZhang2015,ChangShaoZhang2015-nodal,ChangShaoZhang2016}.
For (signless) 1-Laplacian, the (dual) Cheeger inequality turns out to be an equality which is different from the discrete Laplacian case. Besides, the multiplicity of the maximum eigenvalue possesses strong relation with the independence number and chromatic number \cite{Zhang2017}. Due to these exciting results,
graph 1-Laplacian seems to perform better on discrete and combinatorial properties of graphs  than graph Laplacian. In a sense, the (signless) 1-Laplacian provides a {\sl zero-homogenous spectral theory} on graph, with which some advantages on graphic features are founded. 

To some extent, the previous explorations indicate that finding a zero-homogenous spectral theory for a combinatorial object may bring new discoveries and better discrete features -- this is the big picture behind our researches. Our motivation is to realize a zero-homogenous spectral theory for simplicial complex, which is a preliminary attempt under the big picture. As a good starting point, signless 1-Laplacian seems to be a suitable and simple choice along this direction -- because we don't need to consider the routine orientations of a simplical complex when we use signless 1-Laplacian rather than  1-Laplacian.

 In present paper, we will mainly focus on bridging the spectrum of signless 1-Laplacian and
 some combinatorial data of complexes including the independence number, chromatic number and clique covering number, etc. The organization of this paper is as follows. Some basic definitions and a preliminary result involving the spectrum of signless 1-Laplacian on complex is shown in the following subsections. The Courant nodal domain theorem (see Theorem \ref{th:nodal-domain}) will be investigated in Section \ref{section:Courant}.
One of the main parts is Section \ref{section:parameter} where one can find strong and close relations between the spectrum of signless 1-Laplacian and some combinatorial parameters of complex (see Theorems \ref{thm:Hoffman}, \ref{th:color} and \ref{th:independent-clique}). Finally, we study the spectrum for 1-Laplacian on complexes constructed via wedges and duplication of motifs in Section \ref{section:construction}.

We highlight here that both Theorem \ref{thm:Hoffman} and Theorem \ref{th:color} provide comparable bounds with Hoffman's bounds on independence number and chromatic number of a simplicial complex \cite{BGP19,Golubev,Hoffman70,Hoffman03,Nikiforov19}. 
Meanwhile, Theorem \ref{th:color} bridge the 1-Laplacian spectrum, chromatic number and the dual Cheeger constant. And interestingly, Theorem \ref{th:independent-clique} shows a variant analog of the Lovasz Sandwich theorem $\alpha\le\Theta\le\theta\le \kappa$ with $\Theta$ the Shannon capacity and $\theta$ the Lovasz theta number \cite{BGP19}.
\subsection{Preliminary and definition}
An abstract simplicial complex $K=(V,S)$ on a finite vertex set $V$ is a collection of subsets of $V$ that  is closed under inclusion, i.e. $F' \subset F \in S\Rightarrow F' \in S$. A $d$-face is an element of cardinality $(d+1)$ in $S$. The collection of all $d$-faces of the simplicial complex $K$ is denoted by $S_d(K)$ and the number of all $d$-faces is denoted by $\#S_d(K)$ or $|S_d(K)|$.\footnote{ Sometimes, we will also use $|\cdot|$ to denote the absolute value. While the readers could easily distinguish the meanings. } Assume that $\bar F=\{v_0,v_1,\ldots,v_{d+1}\}$ is a $(d+1)$-face of a complex $K$, then
$F_i=\{v_0,\ldots,\hat{v}_i,\ldots,v_{d+1}\}$
is a $d$-face of $\bar F$ which can be denoted by $F_i\prec \bar F$. The boundary set $\partial \bar F $ can be divided into $(d+2)$ boundary faces, $\{F_i\}^{i=d+1}_{i=0}$. For two $d$-faces $F_i$ and $F_j$ sharing a $(d-1)$-face, we say they are $(d-1)$-down neighbours
and denote by $F_i\mathop{\sim}\limits^{\text{down}} F_j$. Similarly, for two $d$-faces $F_i$ and $F_j$ which are boundaries of  a $(d+1)$-simplex, we use the term of
$(d+1)$-up neighbours and write as
$F_i\mathop{\sim}\limits^{\text{up}} F_j$.

According to the combinatorial structure related to up or down adjacency, we define
signless up  1-Laplacian and signless down 1-Laplacian respectively.
 \textbf{Henceforth, we omit the word `signless' for simplicity, that is, we will use (up and down) 1-Laplacian instead of signless 1-Laplacian}.

Let $p=\# S_{d+1}(K)$, $q=\# S_{d}(K)$ and $l=\# S_{d-1}(K)$. The incidence matrix corresponding to $S_{d+1}(K)$ and $S_d(K)$ is defined by $B_d^{up}=(b_{\bar F_i F_j}^{up})_{p\times q}$, where
$$b_{\bar F_i F_j}^{up}=
\begin{cases}
1,&\text{ if }F_j\in \partial \bar F_i, \\
0,&\text{ if }F_j\not\in \partial \bar F_i.
\end{cases}$$
Similarly, the incidence matrix corresponding to $S_d(K)$ and $S_{d-1}(K)$ is $B_d^{down}=(b_{  E_jF_i}^{down})_{l\times q}$, where
$$b_{ E_jF_i}^{down}=
\begin{cases}
1,&\text{ if }E_j\in \partial  F_i, \\
0,&\text{ if }E_j\not\in \partial  F_i.
\end{cases}$$

For any $f=(f_1,\cdots,f_q)\in\mathbb{R}^q$, define the set valued mapping $\Sgn: \mathbb{R}^q\rightarrow 2^{\mathbb{R}^q}$ by
$$\Sgn(f)=\left\{(v_1,\cdots,v_q)^T\in\mathbb{R}^q:v_i\in \Sgn(f_i),\,i=1,\cdots,q\right\}$$
i.e.,
$$
\Sgn(f)=\Sgn(f_1)\times\Sgn(f_2)\times\cdots\times\Sgn(f_q)
$$
 where $2^{\mathbb{R}^q}$ is the power set of $\mathbb{R}^q$ and
$$
\Sgn (t)=\begin{cases}
\{1\}, &\;\text{if }t>0,\\
\{-1\},&\;\text{if }t<0,\\
[-1,1],&\;\text{if }t=0.
\end{cases}
$$

Then define the up and down 1-Laplacian respectively as follows:
\begin{align*}
\Delta_{1,d}^{up} f&=(B_d^{up})^{\top}\Sgn(B_d^{up}f), \\
\Delta_{1,d}^{down} f&=(B_d^{down})^{\top}\Sgn(B_d^{down}f),
\end{align*}
where $f=(f_1,f_2,\cdots,f_q)$ is a real $q$-dimensional vector with components  corresponding to the function values on every $d$-face, i.e., $f_i:=f(F_i)$. And $B^T\Sgn(Bf)$ is defined as $\{B^Tg:g\in \Sgn(Bf)\}$ in which $B=B_d^{up}$ or $B_d^{down}$.  Hereafter, for convenience, with some abuse of notation, $f_i$ also represents the $d$-face 
with subscript $i$.

\begin{remark}
If we remove the `$\Sgn$' in the definition of (signless) 1-Laplacians $\Delta_{1,d}^{up}$ and $\Delta_{1,d}^{down}$, then we can get the signless Laplacian defined on a simplicial complex \cite{KO17}. This can be verified in the following way:  note that $(B_d^{up})^{\top}B_d^{up}f=(D_d^{up}+A_d^{up})f$, where $D_d^{up}=\mathrm{diag}(\deg_1^{up},\ldots,\deg_q^{up})$ with $\deg_i^{up}$ is the number of $(d+1)$-faces containing $f_i$ in boundary, and $A_d^{up}=(a_{jj'})_{q\times q}$ is the adjacent matrix defined as
$$a_{jj'}=\begin{cases}
1,&\text{ if } F_j\cup F_{j'}\in S_{d+1}(K),\\
0,&\text{ otherwise}.
\end{cases}$$
If we adopt the more general form
$$b_{\bar F_i F_j}^{up}=
\begin{cases}
\sqrt{w_i},&\text{ if }F_j\in \partial \bar F_i, \\
0,&\text{ if }F_j\not\in \partial \bar F_i.
\end{cases}$$
where $w_i$ is a nonnegative weight on $\bar F_i\in S_{d+1}(K)$, 
then by standard normalization, we arrive at the (normalized) signless Laplacian
$$\frac{1}{d+2}(D_d^{up})^{-1/2}(B_d^{up})^{\top}B_d^{up}(D_d^{up})^{-1/2}
=\frac{1}{d+2}(I+(D_d^{up})^{-1/2}A_d^{up}(D_d^{up})^{-1/2}),$$
which is called the upper random walk on $d$-simplices defined by a transition probability matrix (Definition 3.1 in \cite{KO17}). For the lower case in \cite{KO17}, we have a similar discussion as above.
\end{remark}

\begin{defn}[eigenvalue problem for up 1-Laplacian]
A pair $(\mu, f)\in \mathbb{R}\times (\mathbb{R}^q\setminus\{0\})$ is called an {\sl eigenpair} of the up 1-Laplacian $\Delta_{1,d}^{up},$ if
\begin{equation}\label{eq:up-1Lap-eigen}
 0\in \Delta_{1,d}^{up}  f - \mu D^{up}\Sgn(f)\;\;(\text{or }\mu D^{up}\Sgn(f)\bigcap \Delta_{1,d}^{up}  f\neq \varnothing),
\end{equation}
where $D^{up}=\mathrm{diag}(\deg_1^{up},\ldots,\deg_q^{up})$ and $\deg_i^{up}$ is the number of $(d+1)$-faces
 containing $f_i$ in boundary. It should be noted that $\Delta_{1,d}^{up}  f - \mu D^{up}\Sgn(f)$ means the Minkowski summation of the vector sets $\Delta_{1,d}^{up}  f$ and $- \mu D^{up}\Sgn(f)$.

A pair $(\mu, f)\in \mathbb{R}\times (\mathbb{R}^q\setminus\{0\})$ is called an {\sl unnormalized eigenpair} of the up 1-Laplacian  $\Delta_{1,d}^{up},$ if
\begin{equation}\label{eq:up-1Lap-eigen-unnormal}
 0\in \Delta_{1,d}^{up}  f - \mu \Sgn(f)\;\;(\text{or }\mu \Sgn(f)\bigcap \Delta_{1,d}^{up}  f\neq \varnothing).
\end{equation}
\end{defn}

\begin{defn}[eigenvalue problem for down 1-Laplacian]
A pair $(\mu,  f)\in \mathbb{R}\times (\mathbb{R}^q\setminus\{0\})$ is called an  {\sl eigenpair} of the down 1-Laplacian $\Delta_{1,d}^{down},$ if
\begin{equation} \label{eq:down-1Lap-eigen}
 0\in \Delta_{1,d}^{down}  f - \mu D^{down}\Sgn(f)\;\;(\text{or }\mu D^{down}\Sgn(f)\bigcap \Delta_{1,d}^{down}  f\neq \varnothing),
\end{equation}
where $D^{down}=\mathrm{diag}(\deg_1^{down},\ldots,\deg_q^{down})$ and $\deg_i^{down}=(d+1)$, i.e., the number of $(d-1)$-faces of $f_i$.

Similarly, a pair $(\mu,f)\in \mathbb{R}\times (\mathbb{R}^q\setminus\{0\})$ is called an {\sl unnormalized eigenpair} of the down 1-Laplacian $\Delta_{1,d}^{down},$ if
\begin{equation} \label{eq:down-1Lap-eigen-unnormal}
 0\in \Delta_{1,d}^{down}  f - \mu\Sgn(f)\;\;(\text{or }\mu\Sgn(f)\bigcap \Delta_{1,d}^{down}  f\neq \varnothing).
\end{equation}
\end{defn}

Direct calculation shows
$$
(\Delta_{1,d}^{up} f)_i=\left\{\left.\sum_{j_1,\ldots,j_{d+1}} z_{ij_1\cdots j_{d+1}}\right|z_{ij_1\cdots j_{d+1}}\in \Sgn(f_i+f_{j_1}+\cdots+f_{j_{d+1}})\right\},
$$
where the summation $\sum_{j_1,\ldots,j_{d+1}} z_{ij_1\cdots j_{d+1}}$ is taken over all $d$-faces that
$f_{j_1},\cdots,f_{j_{d+1}}$ and $f_i$ are in a same $(d+1)$-face. Moreover, $z_{ij_1\cdots j_{d+1}}$ is symmetric on its indices. In coordinate form, the eigenvalue problem \eqref{eq:up-1Lap-eigen} for up 1-Laplacian is to solve $\mu\in\mathbb{R}$ and $f\in \mathbb{R}^q\setminus\{0\}$ such that there exists $z_{ij_1\cdots j_{d+1}}$ satisfying
\begin{equation}\label{eq:eigen-problem-up1-Laplacian}
\sum_{j_1,\ldots,j_{d+1}} z_{ij_1\cdots j_{d+1}}\in \mu \deg_i^{up} \Sgn(f_i),\; i=1,2,\ldots,q.
\end{equation}

Similarly, the coordinate form of down 1-Laplacian reads as
$$
(\Delta_{1,d}^{down} f)_i=\left\{\left.\sum_{i_1,\cdots,i_m} z_{ii_1\cdots i_m}\right|z_{ii_1\cdots i_m}\in \Sgn(f_i+f_{i_1}+...+f_{i_m})\right\}
$$
where the summation $\sum_{i_1,\cdots,i_m}z_{ii_1\cdots i_m}$ is taken over all $d$-faces that $f_{i_1},\ldots,f_{i_m}$ and $f_i$ sharing a same  $(d-1)$-face. Moreover, $z_{ii_1\cdots i_m}$ is symmetric on its indices. The coordinate form of eigenvalue problem for down 1-Laplacian \eqref{eq:down-1Lap-eigen} is:
\begin{equation}\label{eq:eigen-problem-down1-Laplacian}
\sum_{i_1,\cdots,i_m}z_{ii_1\cdots i_m}\in \mu (d+1) \Sgn(f_i),\; i=1,2,\ldots,q.
\end{equation}


From the variational point of view,  $\Delta_{1,d}^{up}$ and
$\Delta_{1,d}^{down}$ are respectively the subdifferential of the following convex functions
$$
I^{up}(f)=\sum_{\bar F \in S_{d+1}(K)} |\sum_{F\prec \bar F} f_F|  =\sum |f_{j_1}+\cdots+f_{j_{d+2}}|
$$
and
$$
I^{down}(f)=\sum_{E \in S_{d-1}(K)} |\sum_{F\succ E} f_F| =
\sum| f_{j_1}+\cdots+ f_{j_m}|,
$$
i.e., $\Delta_{1,d}^{up}f=\partial I^{up}(f)$ and $\Delta_{1,d}^{down} f=\partial I^{down}(f)$. Indeed, the eigenvalue problem \eqref{eq:up-1Lap-eigen} for $\Delta_{1,d}^{up}$ (resp. \eqref{eq:down-1Lap-eigen} for
$\Delta_{1,d}^{down}$) could be derived from the variational problem of $I^{up}(\cdot)$ (resp.  $I^{down}(\cdot)$) on the piecewise linear manifold determined by $\|f\|^{up}=1$  (resp. $\|f\|^{down}=1$), where $\|f\|^{up}=\sum_{i=1}^q \deg_i^{up}|f_i|$ (resp. $\|f\|^{down}=\sum_{i=1}^q \deg_i^{down}|f_i|$).

It is known that critical points of $I^{up}(\cdot)/\|\cdot\|^{up}$ (resp., $I^{down}(\cdot)/\|\cdot\|^{down}$) must be eigenvectors of  $\Delta_{1,d}^{up}$ (resp., $\Delta_{1,d}^{down}$), and the minimal and maximal eigenvalues of $\Delta_{1,d}^{up}$ (resp., $\Delta_{1,d}^{down}$) are respectively the minimum and maximum of $I^{up}(\cdot)/\|\cdot\|^{up}$ (resp., $I^{down}(\cdot)/\|\cdot\|^{down}$). Moreover, the number of eigenvalues of  $\Delta_{1,d}^{up}$ (resp., $\Delta_{1,d}^{down}$) is finite. The proof is based on the standard and routine calculations similar to the related results in \cite{Chang2015,ChangShaoZhang2015,ChangShaoZhang2016} and thus we omit it.

\subsection{A glimpse of the spectrum of signless 1-Laplacian on complexes}

For a graph, we know that $0$ is an eigenvalue of the signless 1-Laplacian if and only if the graph has a bipartite component. As a related analog on simplicial complex, we may first look at the following proposition:

\begin{pro}\label{prop:eigenvalue0} $\Delta_{1,d}^{up}$ has eigenvalue $0$ if one of the following conditions holds.

\begin{enumerate}
\item[(1)]  $\# S_d(K)> \# S_{d+1}(K)$; 

\item[(2)]  the up-degree of each $d$-face is less than or equal to $(d+2)$ and there is at least one $d$-face with the degree less than $(d+2)$; 

\item[(3)]  $S_d(K)$ is $(d+2)$ colorable, i.e., the colors of faces of each $(d+1)$-dim simplex are pairwise different.
\end{enumerate}
\end{pro}

\begin{proof}
(1) It is clear that $0$ is an eigenvalue if and only if $I^{up}(f)=0$ has a nonzero solution $f$, i.e., $f_{j_1}+\cdots+f_{j_{d+2}}=0$, $\forall \bar F_j \in S_{d+1}(K)$.  Since the number of coordinate components of $f$ is $\# S_d(K)$ but the number of absolute terms of $I^{up}$ is $\# S_{d+1}(K)$, and by $\# S_d(K)> \# S_{d+1}(K)$, the related linear system of equations has nonzero solution $f$. So there exists an eigenvector corresponding to $0$.

(2) Since
$$
(d+2)\#S_{d+1}(K)=\sum_{i\in S_d(K)}\deg_i^{up}
<  (d+2) \#S_d(K)
$$
implies $\# S_d(K)> \# S_{d+1}(K)$,
it follows from (1) that (2) holds.

(3) Let $c:S_d(K)\to \{1,\ldots,d+2\}$ be a coloring map such that $\{c(f_{j_1}),\ldots,c(f_{j_{d+2}})\}=\{1,\ldots,d+2\}$ whenever $j_1,\ldots,j_{d+2}$ are $d$-dimensional  faces of a $(d+1)$-dimensional simplex. Letting $$f_i=\begin{cases} 1,& \text { if } c(f_i)=1,\\ -1,& \text { if } c(f_i)=2,\\ 0,& \text{ otherwise,}
\end{cases}$$
for $i=1,\ldots,q$, we have $f_{j_1}+\cdots+f_{j_{d+2}}=0$ whenever $j_1,\ldots,j_{d+2}$ are $d$-dimensional faces of a $(d+1)$-dimensional simplex. Thus, one can take $\mu=0$ and $z_{ij_1\cdots j_{d+1}}=0$ in \eqref{eq:eigen-problem-up1-Laplacian}. Consequently, $0$ is an eigenvalue of $\Delta_{1,d}^{up}$. 
\end{proof}

\begin{remark}
Proposition \ref{prop:eigenvalue0} only involves the combinatorial properties of a simplicial complex, since our definition of signless 1-Laplacian doesn't involve the orientation of  a simplicial complex. For example,  
let $K$ be determined by its facets $\{0,1,2,3\}$, $\{1,2,3,4\}$ and $\{0,2,3,4\}$. Then $S_2(K)$ is $4$-colorable\footnote{Indeed, all the conditions (1),(2),(3) in Proposition \ref{prop:eigenvalue0} hold.}, but $\{0,1,2,3\}$ and $\{1,2,3,4\}$  induce the different orientation on $\{1,2,3\}$.
\end{remark}

\begin{example}
We show the detailed computation of spectrum of signless 1-Laplacian on the tetrahedron.
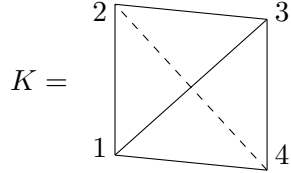
\begin{center}
\begin{tikzpicture}
\draw (0,0) to (0,2);
\draw (0,0) to (2,1.8);
\draw (0,0) to (2,-0.2);
\draw (2,1.8) to (2,-0.2);
\draw (2,1.8) to (0,2);
\draw[dashed] (2,-0.2) to (0,2);
\node (K) at (-1,1) {$K=$};
\node (1) at (-0.2,0.1) {$1$};
\node (2) at (-0.2,1.9) {$2$};
\node (3) at (2.2,1.9) {$3$};
\node (4) at (2.2,0) {$4$};
\end{tikzpicture}
\end{center}

Consider the signless up 1-Laplacian for  $2$-faces $\{(123),(124),(134),(234)\}$.
The eigenvalue problem is to solve the pair $(\mu,f)\in \mathbb{R}\times (\mathbb{R}^n\setminus\{0\})$ such that
$$
\begin{cases}
\exists  z_{1234}\in \Sgn(f_{123}+f_{124}+f_{134}+f_{234}),\\
\text{ and }z_{123}\in \Sgn(f_{123}),
z_{124}\in \Sgn(f_{124}),
z_{134}\in \Sgn(f_{134}),
z_{234}\in \Sgn(f_{234}), \text{ s.t. }\\
z_{1234}=\mu z_{123}=\mu z_{124}=\mu z_{134}=\mu z_{234}.
\end{cases}
$$
We first check $\mu=0$. In this situation, $z_{1234}=0$, which implies that $f_{123}+f_{124}+f_{134}+f_{234}=0$. And $f= (1,-1,1,-1)$ is a solution.

Now we assume that $\mu\ne 0$. Since $f$ has a nonzero component, we can assume without loss of generality that $f_{123}> 0$, which implies that $z_{123}=1$. Note that $\mu=\mu z_{123}=\mu z_{124}$, which follows that $z_{124}=1$ and thus $f_{124}\ge 0$. Similarly, we have $f_{134}\ge 0$ and $f_{234}\ge0$. So, $f_{123}+f_{124}+f_{134}+f_{234}>0$ and then $z_{1234}=1$. Consequently, $1=z_{1234}=\mu z_{123}=\mu$. And $f=(1,0,0,0)$ is a solution. Therefore, $\mu=1$ is the eigenvalue.

In summary, the spectrum of signless 1-Laplacian for $2$-faces of $K$ is $\{0,1\}$.
\end{example}

\begin{remark}
In fact, the above example could be generalized to $\Delta_{1,d-1}^{up}$ for $d$-dimensional simplex. That is, the spectrum of $\Delta_{1,d-1}^{up}$ for $d$-dimensional simplex is  $\{0,1\}$.
\end{remark}

\section{Courant nodal domain theorem}
\label{section:Courant}

In this section, we develop Courant nodal domain theorem for 1-Laplacian on complexes.

Similar to the reason (see \cite{ChangShaoZhang2016} Page 8), we should modify the definition of nodal domain as follows.

\begin{defn}

A set of $d$-simplexes is \sl{up-connected} (resp., \sl{down-connected}), if for any $\sigma$ and $\sigma'$ in such set, there exists a sequence of $\{\sigma_i\}_{i=0}^{i=m}$ such that $\sigma=\sigma_0\mathop{\sim}\limits^{\text{up}} \sigma_1\mathop{\sim}\limits^{\text{up}}\cdots\mathop{\sim}\limits^{\text{up}} \sigma_m=\sigma'$ (resp., $\sigma=\sigma_0\mathop{\sim}\limits^{\text{down}} \sigma_1\mathop{\sim}\limits^{\text{down}}\cdots\mathop{\sim}\limits^{\text{down}} \sigma_m=\sigma'$). 
\end{defn}

\begin{defn}\label{def:nd}
The (up-/down-) nodal domains of an eigenvector $  f=(f_1, f_2, \cdots, f_q)$ of $\Delta_1$
are defined to be maximal (up-/down-) connected components of the support set $D(  f):=\{i: f_i\ne 0\}$.
\end{defn}

\begin{pro}\label{pro:ternarynu}
Suppose $(\mu,  f)$ is an eigenpair of the signless (up-/down-)
1-Laplacian and $D_1,\cdots,D_k$ are (up-/down-) nodal domains of $  f$. Let
$  f^i$ be defined as
\[
f^i_j=
\begin{cases}
\frac{f_j}{\sum_{j\in D_i(  f)}\deg_j|f_j|}, & \text{ if } \;\ j\in D_i(  f), \\
0, & \text{ if } \;\ j\not\in D_i(  f),\\
\end{cases}
\]
for $j=1,2,\cdots,q$. Then $(\mu,  f^i)$ is an eigenpair, $i=1,2,\ldots,k$.
\end{pro}
\begin{proof}
It can be directly verified that $i\mathop{\sim}\limits^{\text{up}}j$ (resp. $i\mathop{\sim}\limits^{\text{down}}j$)  if and only if $f_i$ and $f_j$ appear in a same term $|\cdot|$ in the summation form of $I^{up}(\cdot)$ (resp. $I^{down}(\cdot)$). Hence, we derive that $\Sgn(f^i_j)\supset \Sgn(f_j)$ and $\Sgn(f^i_{j}+f^i_{j_1}+\cdots+f^i_{j_m}) \supset  \Sgn(f_{j}+f_{j_1}+\cdots+f_{j_m})$, $j'\sim j$, $j=1,2,\cdots,n$, $i=1,2,\cdots,k$. Then, using the coordinate form of 1-Laplacian eigenvalue problems \eqref{eq:eigen-problem-up1-Laplacian} and \eqref{eq:eigen-problem-down1-Laplacian}, we complete the proof. 
\end{proof}

Now, fix the dimension $d$ and $n:=\# S_d(K)$, and denote $I(\cdot)=I^{up}(\cdot)$ (resp. $I^{down}(\cdot)$) and $\|\cdot\|=\|\cdot\|^{up}$ (resp. $\|\cdot\|^{down}$). We apply the Liusternik-Schnirelmann theory to $\Delta_1 $. Note that $I(f)$ is even, and $X=\{f : \|f\|=1\}$ is symmetric. For a symmetric set $T\subset X$, i.e., $-T=T$, the Krasnoselski genus (see \cite{Chang1985,Rabinowitz1986}) of $T$ is defined to be
\begin{equation*}
\gamma(T) =
\begin{cases}
0, & \text{if}\; T=\varnothing,\\
\min\limits\{k\in\mathbb{Z}: \exists\; \text{odd continuous}\; h: T\to \mathbb{S}^{k-1}\}, & \text{otherwise}.
\end{cases}
\end{equation*}
Obviously, the genus is a topological invariant. Let us define
\begin{equation}
c_k=\inf_{\gamma(T)\ge k} \max_{  f\in T\subset X} I(f)=\inf_{\gamma(T)\ge k} \max_{  f\in T\subset \mathbb{R}^n\setminus \{0\}} \frac{I(f)}{\|f\|},\quad k=1,2, \ldots n.
\end{equation}
By the same way as already used in \cite{Chang2015}, it can be proved that these $c_k$ are critical values of $I(f)$. One has
$$c_1\leq c_2 \leq \cdots \leq c_n,$$
and if $0\le\cdots\le c_{k-1}< c_k=\cdots=c_{k+r-1}< c_{k+r}\leq \cdots\leq 1$, the multiplicity of $c_k$ is defined to be $r$. The Courant nodal domain theorem for the signless $1$-Laplacian reads
\begin{theorem}\label{th:nodal-domain}
Let $c_k$ be an eigenvalue with multiplicity $r$ and let $f^k$ be an eigenvector associated with $c_k$.  Then
$$1\leq  S(  f^k)\leq k+r-1,$$
where $S(  f^k)$ is the number of nodal domains of $f^k$.
\end{theorem}

\begin{proof}
Assume there are $n$ $d$-faces in the complex.  Suppose the contrary, that there exists an eigenvector $f^k=(f_1, f_2, \cdots, f_n)$ corresponding to the variational eigenvalue $c_k$ with multiplicity $r$ such that $S(  f^k)\ge k+r$. Let $D_1(  f^k)$, $\cdots$, $D_{k+r}(  f^k)$ be the nodal domains of $  f^k$. Let $  g^i=(g^i_1, g^i_2, \cdots, g^i_n)$, where
\[
g^i_j=
\begin{cases}
\frac{f_j}{\sum_{j\in D_i(  f)}\deg_j|f_j|}, & \text{ if } \;\ j\in D_i(  f), \\
0, & \text{ if } \;\ j\not\in D_i(  f),\\
\end{cases}
\]
for $i=1,2,\cdots, k+r$, $j=1,2,\cdots,n$. By the construction of $g^i$, $i=1,2,\cdots, k+r$, we have:

(1) The nodal domain of $  g^i$ is the $i$-th nodal domain of
$  f^k$, i.e., $D(  g^i)=D_i(  f^k)$;

(2) $D(  g^i)\cap D(  g^j)=\varnothing$, $i\ne j$;

(3) By Proposition \ref{pro:ternarynu}, $  g^1,\cdots,  g^{k+r}$ are all eigenvectors with the same eigenvalue $c_k$.

Now, for any $f\in\mathrm{span}(  g^1,\cdots,  g^{k+r})\cap X $, there exist $a_1,\ldots,a_{k+r}$ such that $f=\sum\limits_{i=1}^{k+r}a_i  g^i\in X$. And for any $l\in \{1,2,\ldots,n\}$, there exists a  unique $j$ such that  $f_l=a_jg^j_l$. Hence, $|f_l|=\sum_{j=1}^{k+r}|a_j||g^j_l|$. Since $  f\in X$, $  g^j\in X$, $j=1,\cdots,k+r$, we have
$$
1=\sum_{l=1}^n\deg_l|f_l|=\sum_{l=1}^n\deg_l\sum_{j=1}^{k+r}|a_j||g^j_l|
=\sum_{j=1}^{k+r}|a_j|\sum_{l=1}^n\deg_l|g^j_l|=\sum_{j=1}^{k+r}|a_j|.
$$
Note that $I(\cdot)$ is convex and even. Therefore,  we have
\begin{align*}
I(  f)&=I(\sum\limits_{i=1}^{k+r}a_i  g^i)
=I(\sum\limits_{i=1}^{k+r}|a_i|\sgn(a_i)  g^i)
\\&\le \sum\limits_{i=1}^{k+r}|a_i|  I(\sgn(a_i)  g^i)
\le \sum\limits_{i=1}^{k+r}|a_i|  I(g^i)
\\&\leq \max\limits_{i=1,2,\cdots,k+r}I(  g^i).
\end{align*}
Note that $  g^1,\cdots,  g^{k+r}$ are non-zero orthogonal vectors, so $\mathrm{span}(  g^1,\cdots,  g^{k+r})$ is a $k+r$ dimensional linear space. It follows that $\mathrm{span}(  g^1,\cdots,  g^{k+r})\cap X$ is a symmetric set which is homeomorphous to $S^{k+r-1}$. Obviously, $\gamma(\mathrm{span}(  g^1,\cdots,  g^{k+r})\cap X)=k+r$. Therefore, we derive that
\begin{align*}
c_{k+r}&=\inf\limits_{\gamma(A)\ge k+r}\sup\limits_{  f\in A} I(  f)
\\&\leq \sup\limits_{  f\in \mathrm{span}(  g^1,\cdots,  g^{k+r})\cap X} I(  f)
\\&=\max\limits_{i=1,\cdots,k+r} I(  g^i)
\\&=c_k,
\end{align*}
It contradicts with $c_k< c_{k+r}$. So the proof is completed.
\end{proof}

\section{On some combinatorial parameters of complex}
\label{section:parameter}
In this section, we concentrate on the relationships between the eigenvalues of signless 1-Laplacian and other attractive parameters, such as chromatic number, independence number and clique covering number.
\subsection{Independence number and chromatic number for vertices}
Firstly, we recall the concepts of independence number and chromatic number of a hypergraph. The definition of chromatic number of hypergraphs generalize chromatic number of graphs in various ways, see for example \cite{KrivelevichSudakov}. The chromatic number is defined for $r$-unifrorm hypergraphs as

$$\chi_s=\min\{k:\exists k\text{-partion }(V_1,\ldots,V_k) \text{ such that } |V_i\cap E|\le s,\forall\text{ edge } E, \forall\,i\},$$
 while the independence number of a hypergraph is defined by
$$\alpha_s=\max\{|A|:|A\cap E|\le s, \forall \text{ edge }E\}$$
for $1\leq s \leq r-1$ (see \cite{KrivelevichSudakov}).
Note that a simplicial complex can be regarded as a hypergraph. Thus, the definition of  independence number and chromatic number for simplicial complexes can be defined as follows:
\begin{defn}[Independence number]
For a simplicial complex $K=(V,S)$, the independence number is
$$\alpha_s=\max\{|A|:|A\cap F|\le s,\; \forall \text{face }F\}.$$
\end{defn}

\begin{defn}[Chromatic number]
For a simplicial complex $K=(V,S)$, the chromatic number is
$$\chi_s=\min\{k:\exists\, k\text{-partion }(V_1,\ldots,V_k) \text{ such that } |V_i\cap F|\le s,\,\forall\text{ face } F, \forall\,i\}.$$
\end{defn}

In \cite{Golubev}, the independence number and chromatic number of simplicial complex are respectively defined by
$$\alpha=\max\{|A|:A\not\supset F, \forall \text{ maximal face }F\}$$
and
 $$\chi=\min\{k:\exists k\text{-partion }(V_1,\ldots,V_k) \text{ such that } V_i\not\supset F,\forall\text{ maximal face } F, \forall \,i\}$$
However, in the proof of main theorems in \cite{Golubev}, the author essentially deals with $\alpha_{d}$ and $\chi_{d}$, where $d$ is the dimension of complex. That is, the results still hold if we replace `$\alpha$' and  `$\chi$' by  `$\alpha_{d}$' and  `$\chi_{d}$' in those theorems. In this subsection, we  will concentrate on $\chi_s$ and $\alpha_s$ and study their relations with eigenvalues of 1-Laplacian.
An elementary result for the relations of these concepts are:
\begin{pro}
For simplicial complex $K=(V,S)$, we have the statements below.
\begin{enumerate}
\item[(1)] $\chi_s\alpha_s \ge |V|$, $\chi_s\le  \lceil \chi_1/s \rceil $,  for all $s$. Here $\lceil x \rceil$ is the ceiling function on $x$, i.e., the minimum integer that does not less than $x$. Moreover, $\chi_s\le \lceil \chi_t/\lfloor s/t \rfloor \rceil$, where $\lfloor x \rfloor$ is the Gauss function on $x$, i.e., the maximum integer that does not exceed $x$.
\item[(2)]  $ \chi\alpha \ge |V|$, $\chi_d\le\chi\le \lceil \chi_1/d\rceil$, $\alpha\le \alpha_{d}$, where $d$ is the dimension of the complex. For a pure (i.e., homogeneous) complex, $\alpha= \alpha_{d}$.
\item[(3)]The above definitions for independence number $\alpha_s$ and chromatic number $\chi_s$ are respectively equivalent to  $$
\alpha'_s=\max\{|A|:|A\cap F|\le s, \forall s\text{-face }F\}
$$
and
$$\chi'_s=\min\{k:\exists k\text{-partion }(V_1,\ldots,V_k) \text{ such that } |V_i\cap F|\le s,\forall\text{ s-face } F, \forall\,i\}.$$
\end{enumerate}
\end{pro}

\begin{proof}~~
\begin{enumerate}
\item[(1)] It is evident that $\chi_s\alpha_s \ge |V|$ by the definitions. To prove $\chi_s\le  \chi_1/s$, let
 $(V_1,\ldots,V_{\chi_1})$ be a partition of $V$ such that $|V_i\cap F|\le 1$, for any face $F$.
Let $B_i=A_{(i-1)s+1}\cup \cdots \cup A_{is}$ for $i=1,\ldots,m-1$, and $B_m=A_{(m-1)s+1}\cup\cdots\cup A_{\chi_1}$, where $m=\lceil \chi_1/s \rceil$. Then $|B_i\cap F|\le \sum_{j=1}^s|A_{(i-1)s+j}\cap F|\le s$, $i=1,\ldots,m$. Therefore, $\chi_s\le m$.

The rest of the proof is similar with the above, but we shall provide the details for reader's convenience. Let
 $(V_1,\ldots,V_{\chi_t})$ be a partition of $V$ such that $|V_i\cap F|\le t$, $\forall F$.
Let $B_i=A_{(i-1)s'+1}\cup \cdots \cup A_{is'}$ for $i=1,\ldots,m-1$, and $B_m=A_{(m-1)s'+1}\cup\cdots\cup A_{\chi_t}$, where $m=\lceil \chi_t/s' \rceil$ and $s'=\lfloor s/t \rfloor$. Then $$|B_i\cap F|\le \sum_{j=1}^{s'}|A_{(i-1)s'+j}\cap F|\le s't= \lfloor s/t \rfloor t\le s,$$ for $i=1,\ldots,m$. Therefore, $\chi_s\le m$.

\item[(2)] Note that $\chi\alpha \ge |V|$ and $\chi\le \lceil \chi_1/d\rceil$ have been proved in \cite{Golubev}.  Now we prove $\alpha\le \alpha_{d}$. Since the dimension of $K$ is $d$, the maximal face $F$ has at most $d+1$ elements. So $F\not\subset A$  implies $|F\cap A|\le |F|-1\le d$, for any maximal face $F$. Thus, $|F\cap A| \le d$ holds for any face $F$. This deduces that
$$\{|A|:A\not\supset F, \forall \text{ maximal face }F\} \subset \{|A|:|A\cap F|\le d, \forall \text{ face }F\}.$$
Therefore,  $\alpha\le \alpha_{d}$. Similarly, $\chi\ge \chi_d$.

A maximal face (i.e., facet) is not need to be a $d$-face. But if the complex is homogeneous (or pure), then facets must  coincide with $d$-faces. So, it is easy to check that $\alpha=\alpha_d$.

\item[(3)] It is easy to check that $\alpha'_s \geq \alpha_s.$  Next we will show the reverse inequality.
Indeed, it is enough to show that for any set $A$ with $|A \cap F|\leq s $ for all $s$-faces $F$, we have $|A \cap F|\leq s$ for any face whose dimension is larger than $s$. For the contrary, there is some $(s+k)$-face $F'$ satisfying $|A \cap F'|\geq s+1$. Then there is one $s$-face whose vertices are
in $A$, which is a contradiction.  The proof for $\chi'_s=\chi_s$ is similar.
\end{enumerate}
\end{proof}

\begin{remark}\label{rem:alpha!=alpha_d}
There exists a complex $K$ such that $\alpha<\alpha_d$, where $d=\mathrm{dim}(K)$. Indeed, let $V=\{1,2,3,4,5\}$ and let

$K=\{\{1\},\{2\},\{3\},\{4\},\{5\},\{1,2\},\{1,3\},\{2,3\},\{1,2,3\},\{1,4\},\{2,5\}\}$. Then $K$ is a complex with dimension $d=2$. It can be verified that $\alpha=3$ (where an independence set can be $\{3,4,5\}$) and $\alpha_d=4$ (where an independent set can be $\{1,2,4,5\}$).
\end{remark}

\begin{theorem}\label{thm:Hoffman}
Let  $K=(V,S)$ be a nonempty $d$-dimensional complex. Then $$\alpha\le\alpha_d\le  \min\left\{|V|(1-\frac{ \mu_{d-1}}{2M_0}),|V|\frac{M_0}{M_0+\mu_{d-1}},|V|\frac{M_0-\mu_0}{M_0+m_0}\right\}$$
and hence
$$\chi\ge\chi_d\ge  \max\left\{\frac{ 2M_0}{2M_0-\mu_{d-1}},1+\frac{\mu_{d-1}}{M_0},\frac{M_0+m_0}{M_0-\mu_0}\right\}$$
where $M_0$ and $m_0$ are the largest and smallest degree of vertices in $K$,
 $\mu_{d-1}$ is the minimum unnormalized eigenvalue of $\Delta^{up}_{1,d-1}$
\footnote{The eigenvalues which we concern are the unnormalized eigenvalues of $\Delta^{up}_{1,d-1}$, i.e., the corresponding constraint relates to $\|\vec f\|_1:=\sum_{(d-1)\text{-face }F}|f_F|$, not the normalized version $\|\vec f\|:=\sum_{(d-1)\text{-face }F}\text{deg}_F|f_F|$}, and similarly $\mu_0$ is the minimum unnormalized eigenvalue of $\Delta^{up}_{1,0}$.
\end{theorem}

\begin{proof}
Let $A\subset V$ be the largest independent set with $|A|=\alpha_d$, and for some fixed $a,b\in \mathbb{R}$, let $$f_i=\begin{cases}a,& i\in A,\\b,& i\in A^c:=V\setminus A.
\end{cases}$$
Then \begin{align*}
\sum_{i\sim j}|f_i+f_j|&=\sum_{i\sim j,i\in A,j\in A^c}|a+b|+\sum_{i\sim j\text { in } A}2|a|+\sum_{i\sim j\text { in } A^c}2|b|
\\&=|a+b|\cdot|E(A,A^c)|+|b|(\vol(A^c)-|E(A,A^c)|) 
\end{align*}
and
$$\sum_{i\in V}|f_i|=|a||A|+|b||A^c|.$$
Thus,
$$\max_{i\in V} \deg^{up}_i=\max_{\vec g\ne \vec 0} \frac{\sum_{i\sim j}|g_i+g_j|}{\sum_{i\in V}|g_i|}\ge  \frac{\sum_{i\sim j}|f_i+f_j|}{\sum_{i\in V}|f_i|}\geq\frac{|a+b|\cdot|E(A,A^c)|}{|a||A|+|b||A^c|}.$$
Taking $a=|A^c|$ and $b=|A|$,  we get
\begin{equation}\label{eq:2-upper-estimate}
\max_{i\in V} \deg^{up}_i\ge |E(A,A^c)|\frac{|V|}{2\alpha_d(|V|-\alpha_d)}.
\end{equation}

Let $M_i$ and $m_i$ be respectively maximum and minimum  number of $(i+1)$-faces that contains an $i$-face, $i\in\{0,1,\ldots,d-1\}$.

Since $A$ is independent, for each $(d-1)$-face $F\subset A$, the number of $d$-faces that contains $F$ and meets $A^c$ is larger than or equals to $m_{d-1}$. And it can be proved that $m_0>m_1>\cdots>m_{d-1}$.


\textbf{Claim}: for any $i$-face $F\subset A$ with $i\in\{0,\ldots, d-1\}$, the number of $(i+1)$-faces that contains $F$ and meets $A^c$ is larger than or equals to $m_{d-1}$.

We prove the above claim by mathematical induction on $i$ from $(d-1)$ to $0$.

For $i=d-1$, the claim holds accroding to definition of $m_{d-1}$ and the independence of $A$.

Suppose the claim  holds for $i$. Then for the case of  $(i-1)$, for any $(i-1)$-face $F$, there are two subcases:
\begin{enumerate}
\item $F$ contains in an $i$-face $F'$ with all its vertices in $A$.

In this case, using the inductive hypothesis for $F'$, there exist $v_1,\ldots,v_{m_{d-1}}\in A^c$ such that $F'\cup \{v_1\}$, $\cdots$, $F'\cup \{v_{m_{d-1}}\}$ are $(i+1)$-faces. In consequence, $F\cup \{v_1\}$, $\cdots$, $F\cup \{v_{m_{d-1}}\}$ are $i$-faces, which means that the claim holds for such $F$.

\item There is no $i$-face with all its vertices in $A$ that contains $F$.

In this case, all $i$-faces that contains $F$ must meet $A^c$, and the number of $i$-faces containing $F$ is at least $m_{i-1}\ge m_{d-1}$. Thus the claim holds.
\end{enumerate}


In particular, for $i=0$, it means that $|E(\{i\},A^c)|\ge m_{d-1}$ for any $i\in A$. Thus, $$|E(A,A^c)|=\sum_{i\in A}|E(\{i\},A^c)|\ge |A| m_{d-1}=\alpha_{d}m_{d-1}.$$


Taking $\vec g=\vec e_{F_0}$, where $F_0$ is a $(d-1)$-face such that $|\{i\in F_0^c: F_0\cup \{i\}\text{ is }d\text{-face}\}|=m_{d-1}$. Then $m_{d-1}=I^{up}(\vec g)/\|\vec g\|_{1}\ge \mu_{d-1}$ and hence,
\begin{equation}\label{eq:1-lower-estimate}
|E(A,A^c)|\ge \alpha_{d}\mu_{d-1}.
\end{equation}

Combining \eqref{eq:2-upper-estimate} and  \eqref{eq:1-lower-estimate}, we have
$$
M_0=\max_{i\in V} \deg^{up}_i\ge \mu_{d-1}\frac{|V|}{2(|V|-\alpha_d)},
$$
which derives our desired inequality.

It is also clear that $\alpha_{d}m_0\le |E(A,A^c)|\le \alpha_{d}M_0$ and $\vol(A^c)\le (|V|-\alpha_{d})M_0$.

Taking $a=-1$ and $b=1$, we obtain
$$\mu_0=\inf_{\vec g\ne \vec 0} \frac{\sum_{i\sim j}|g_i+g_j|}{\sum_{i\in V}|g_i|}\le  \frac{\sum_{i\sim j}|f_i+f_j|}{\sum_{i\in V}|f_i|}= \frac{\vol(A^c)-|E(A,A^c)|}{|V|}\le \frac{(|V|-\alpha_{d})M_0-\alpha_{d}m_0}{|V|},$$
which implies $\alpha_{d}\le |V|\frac{M_0-\mu_0}{M_0+m_0}$.

Note that
$$\max_{ab\ne 0} \frac{|a+b|}{|a||A|+|b||A^c|}=\max\{\frac{1}{|A|},\frac{1}{|A^c|}\}=\max\{\frac{1}{\alpha_d},\frac{1}{|V|-\alpha_d}\}.$$
This implies
\begin{equation}\label{eq:1-upper-estimate}
\max_{i\in V} \deg^{up}_i\ge |E(A,A^c)|\max\{\frac{1}{\alpha_d},\frac{1}{|V|-\alpha_d}\}.
\end{equation}
Combining \eqref{eq:1-upper-estimate} and  \eqref{eq:1-lower-estimate}, we have
$$
M_0=\max_{i\in V} \deg^{up}_i\ge \mu_{d-1}\max\{1,\frac{\alpha_d}{|V|-\alpha_d}\},
$$
which derives $\alpha_d\le \frac{M_0}{M_0+\mu_{d-1}}|V|$.
\end{proof}

\begin{remark}Here the proof of Theorem \ref{thm:Hoffman} is inspired by the technique of the proof of Theorem 2 in \cite{Golubev}.

Since $A$ is the maximal independent set,  for any $j\in A^c$, $A\cup \{j\}$ is not independent, i.e., there is a $d$-face in $A\cup \{j\}$ containing $j$ as its vertex, which implies that $|E(A,\{j\})|\ge d$.  Thus,
\begin{equation}\label{eq:2-lower-estimate}
|E(A,A^c)|=\sum_{j\in A^c}|E(A,\{j\})|\ge |A^c|d=(|V|-\alpha_{d})d.
\end{equation}
Combining \eqref{eq:1-upper-estimate} with  \eqref{eq:2-lower-estimate}, we have
$$\alpha_d\ge |V|\frac{d}{M_0+d}.$$
Combining \eqref{eq:2-upper-estimate} with  \eqref{eq:2-lower-estimate}, we have
$$\alpha_d\ge |V|\frac{d}{2M_0}.$$
\end{remark}

\begin{remark}
In the one-dimensional case, i.e., the graph case, we note Hoffman's result $\alpha\le |V|\frac{\lambda_{\max}-m_0}{\lambda_{\max}}$, where $\lambda_{\max}$ is the largest eigenvalue of the Laplacian, $m_0$ is the minimal degree.  For a $5$-order cyclical graph $G=(V,E)$ with $V=\{1,2,3,4,5\}$ and $E=\{\{1,2\},\{2,3\},\{3,4\},\{4,5\},\{5,1\}\}$,  $\lambda_{\max}(G)=2-2\cos\frac45\pi\approx 3.618$ and thus Hoffman's upper bound $\approx 5\frac{3.618-2}{3.618}=2.23$. Note that the upper bound in Theorem \ref{thm:Hoffman} is $5\frac{2-\frac25}{4}=2$, which is better than Hoffman's upper bound.

Since Theorem 1.2 in \cite{Golubev} coincides with Hoffman's bound when the simplicial complex is 1-dimensional, 
it is clear that Theorem \ref{thm:Hoffman} could provide better bounds on some special cases. That is, Theorem \ref{thm:Hoffman} can be comparable to some kinds of Hoffman's bound \cite{Hoffman70,Hoffman03} like Theorem 1.2 in \cite{Golubev}.
\end{remark}

So far, we focused on vertices of simplicial complexes, i.e. $0$-faces. In fact, we can also define independence number and chromatic number for any $i$-face of simplicial complexes. In the following subsection, we will give corresponding definitions and study the relationships between eigenvalues of 1-Laplacian and them.

\subsection{Chromatic number for $i$-faces}

\begin{defn}[chromatic number]
A $d$-face coloring of a simplicial complex assigns a color to each $d$-face so that no two faces that contain in the same $(d+1)$-face have the same color. The smallest number of colors needed is called its chromatic (or coloring) number.
\end{defn}

\begin{defn}[dual Cheeger constant]
The dual Cheeger constant on the set $S_d$ of $d$-faces is defined as
$$
h^{up*}=\max\limits_{A\cap B=\varnothing\ne A\cup B\subset S_d}\frac{2|E_{S_d}(A,B)|}{\vol(A)+\vol(B)},
$$
where $E_{S_d}(A,B)$ it the collection of all $(d+1)$-simplices that has some $d$-face in $A$ and also has some $d$-face in $B$.
\end{defn}

One can easily see that the chromatic number of all $d$-faces is at least $d+2$. 

\begin{theorem}\label{th:color}
Let $\chi$ be the chromatic number of $S_d$ (the set of d-faces) with respect to the up-adjacent relation. Assume $\vol:S_d\to (0,+\infty)$ is a given degree function on $S_d$  and $\vol(S):=\sum_{i\in S}\vol(i)$ is the volume of $S$, for any $S\subset S_d$. Let $\|f\|=\sum_{i\in S_d}\vol(i)|f_i|$ for any $f\in \mathbb{R}^{S_d}$. Then
\[\mu^{up}_1\le1-\frac{2(d+1)}{d+\chi}\;\;\text{ and }\;\; \chi\ge \frac{2(d+1)}{h^{up*}}-d .\]
where 
$\mu^{up}_1=\min\limits_{f\ne0} \frac{I^{up}(f)}{\|f\|}$.
Furthermore, these bounds are sharp.
\end{theorem}

\begin{proof}[Proof of Theorem \ref{th:color}]
Let $S^{1}_d,\cdots,S_d^{\chi}$ be the color classes of $S_d$. 
Given an integer $k\in\{1,2,\ldots,\chi\}$, we  define the vector $f$ by
\[
f_i=\begin{cases}
a,& \text{ if } i\in S_d^k,\\
b,& \text{ if } i\not\in S_d^k.
\end{cases}
\]
We have
\[\mu^{up}_1\|f\|\leq I^{up}(f).\]
It is evident that
\[\|f\|=|a|\vol(S_d^k)+|b|\vol(S_d\setminus S_d^k)=(|a|-|b|)\vol(S_d^k)+|b|\vol(S_d),\]
and
\begin{align*}
I^{up}(f)&=|a+(d+1)b|e_{0,d}(S_d^k,S_d\setminus S_d^k)+(d+2)|b|e_{d+1}(S_d\setminus S_d^k)
\\&=(|a+(d+1)b|-(d+2)|b|)e_{0,d}(S_d^k,S_d\setminus S_d^k)+(d+2)|b|e_{d+1}(S_d),
\end{align*}
where $e_{0,d}(S_d^k,S_d\setminus S_d^k)$ counts the number of $(d+1)$-simplexes with one $d$-face in $S_d^k$ and others in $S_d\setminus S_d^k$, and $e_{d+1}(S_d\setminus S_d^k)$ (resp. $e_{d+1}(S_d)$) is the number of $(d+1)$-simplexes with $d$-faces in $S_d\setminus S_d^k$ (resp. $S_d$).

In summary, for every $k=1,\cdots,\chi$, we have
\begin{equation}\label{eq:c|x|I(x)}
\mu^{up}_1((|a|-|b|)\vol(S_d^k)+|b|\vol(S_d))\leq (|a+(d+1)b|-(d+2)|b|)e_{0,d}(S_d^k,S_d\setminus S_d^k)+(d+2)|b|e_{d+1}(S_d).
\end{equation}
Summing these inequalities for $k=1,2,\cdots,\chi$, we obtain
\[\mu^{up}_1\sum_{k=1}^\chi((|a|-|b|)\vol(S_d^k)+|b|\vol(S_d))\leq \sum_{k=1}^\chi((|a+(d+1)b|-(d+2)|b|)e_{0,d}(S_d^k,S_d\setminus S_d^k)+(d+2)|b|e_{d+1}(S_d)).\]
Elementary computation gives
\[
\mu^{up}_1((|a|-|b|)\vol(S_d)+|b|\vol(S_d)\chi)\leq (|a+(d+1)b|-(d+2)|b|)(d+2)e_{d+1}(S_d)+(d+2)|b|e_{d+1}(S_d)\chi
\]
Now we get
\begin{align*}
\mu^{up}_1&\le\frac{(d+2)e_{d+1}(S_d)}{\vol(S_d)}\frac{|a+(d+1)b|-(d+2)|b|+|b|\chi}{|a|-|b|+|b|\chi}
\\&=\frac{(d+2)e_{d+1}(S_d)}{\vol(S_d)}\left(1-\frac{|a|+(d+1)|b|-|a+(d+1)b|}{|a|+(\chi-1)|b|}\right)
\end{align*}
It is easy to see that
\[
\max_{(a,b)\ne(0,0)} \frac{|a|+(d+1)|b|-|a+(d+1)b|}{|a|+(\chi-1)|b|}
=\max_{t\leq 0}\frac{|t|+(d+1)-|t+(d+1)|}{|t|+(\chi-1)}=\frac{2(d+1)}{d+\chi}
\]
where the maximum arrives at $t=-d-1$,
and this implies the desired inequality
\[\mu^{up}_1\le\frac{(d+2)e_{d+1}(S_d)}{\vol(S_d)}(1-\frac{2(d+1)}{d+\chi}).\]
Since the elementary fact shows $(d+2)e_{d+1}(S_d)=\vol(S_d)$, we have
\[\mu^{up}_1\le 1-\frac{2(d+1)}{d+\chi}.\]

Next we prove that the bound is sharp. 
 In fact, if $K$ is a $(d+1)$-simplex, then the equality holds.

 Finally, using the equality\footnote{This relation can be proved in the same way as shown in \cite{ChangShaoZhang2016}.} $\mu^{up}_1+h^{up*}=1$, we get $\frac{2(d+1)}{d+\chi}\le h^{up*}$ and thus $\chi\ge \frac{2(d+1)}{h^{up*}}-d $.
\end{proof}

\begin{remark}
It is worth noting that Proposition \ref{prop:eigenvalue0} (3) is a special case of  Theorem \ref{th:color} by taking $\chi=d+2$.
\end{remark}


\begin{cor}\label{th:color-c1}
Let $\chi$ be the chromatic number of $G$, then
\[\mu_1^+\le 1-\frac{2}{\chi}\;\;\text{ and }\;\;\frac{2}{h^*}\le\chi.\]
\end{cor}

\begin{remark}
For the normalized signless Laplacian on a graph, we similarly have
$\lambda_1^+\le 1-\frac{1}{\chi-1}$. And combining with Corollary  \ref{th:color-c1}, we obtain
$$\chi\ge \max\left\{\frac{2}{1-\mu_1^+},\frac{1}{1-\lambda_1^+}+1\right\}.$$
Interestingly, by dual Cheeger inequality, there is an inequality $\frac{\lambda_1^+}{2}\le \mu_1^+\le \sqrt{\lambda_1^+(2-\lambda_1^+)}$, where $\lambda_1^+$ is the smallest eigenvalue of signless Lapalacian and $\mu_1^+$ is the minimal eigenvalue of signless 1-Lapalacian.

 Hoffman \cite{Hoffman70} has the following estimate $\chi\ge 1-\frac{\lambda_{\max}(A)}{\lambda_{\min}(A)}$ of chromatic number using adjacent matrix $A$ of a graph. Some update results can be found in  \cite{WE13}. Moreover, in \cite{Apkarian2013}, there is an estimate $$\chi\ge 1+\frac{1}{\lambda_{\max}(L)-1},$$
which coincides with the inequality $\chi\ge 1+\frac{1}{1-\lambda_1^+}$ as before.

On a cyclical graph with $n$ vertices, Corollary  \ref{th:color-c1} gives better estimate of the chromatic number than Hoffman's bound. For example, let's look at an odd cyclical graph, a ring with $(2m+1)$ vertices, with $m\ge 2$. The largest Laplacian eigenvalue $\lambda_{\max}(L)=1-\cos\frac{2m\pi}{2m+1}$, while $\mu_1^+=1-\frac{2m}{2m+1}$. Then Hoffman's lower bound for chromatic number is $1-\frac{1}{\cos\frac{2m\pi}{2m+1}}$ which is smaller than our lower bound $\frac{2}{1-\mu_1^+}=2+\frac{1}{m}$.
\end{remark}


\subsection{Independence number and clique covering number for $i$-faces}

Similar to the results in \cite{Zhang2017}, we give the  counterpart of connections between the independence number and clique covering number of $K$ as well as the multiplicity of the eigenvalue $1$ of $\Delta_1$. The definitions and notions are listed below.

\begin{itemize}
\item $\gamma$: topological multiplicity of the maximal eigenvalue $1$ of $\Delta_1$.

\item $t$: times of $1$ appearing in the sequence of variational eigenvalues $(c_k)_{k=1}^n$.

\item $\alpha$: independence number of $S_d$, i.e., the cardinality of the largest subset of $d$-faces that does not adjacent. It is defined by $\alpha=\max\{p:\text{there exist }p\text{ faces}$ $\text{in }S_d$ $\text{which are pairwise non-adjacent}\}$.

\item $\kappa$: the clique covering number (the smallest number of cliques of $S_d$ whose union covers $S_d$). Here a clique is a subset of $S_d$ such that any two $d$-faces in such clique are adjacent.
\end{itemize}

\begin{theorem}\label{th:independent-clique}
\begin{eqnarray}
\alpha\leq t\leq \gamma\leq \kappa.
\end{eqnarray}
\end{theorem}

\begin{proof}
$t\leq \gamma$ is a basic result \cite{Chang2015}. So it suffices to prove $\alpha\leq t$ and $\gamma\leq \kappa$.

$\alpha\leq t$: We only need to show that $c_k=1$ whenever $k\ge n-\alpha+1$, where $n:=\#S_d$.

Since $n-k+1\leq \alpha$, there exist $n-k+1$ non-adjacent $d$-faces in $S_d$, denoted by $i_1,\cdots,i_{n-k+1}$. Set $X_{n-k+1}=X\cap \Span( e_{i_1},\cdots, e_{i_{n-k+1}})$, where $ e_1=(1,0,0,\cdots,0), e_2=(0,1,0,\cdots,0),\cdots, e_n=(0,0,\cdots,0,1)\in \mathbb{R}^n$.
For each $ f$ in $X_{n-k+1}$, there exists some $(l_1,\cdots,l_{n-k+1})\ne  0$ such that $ f=l_1 e_{i_1}+\cdots+l_{n-k+1} e_{i_{n-k+1}}$.
It should be noted that $i\not\sim j$ if $f_if_j\ne0$. Thus, each nodal domain of $ f$ must be a singleton set, which implies that $I( f)=1$.
Therefore, it follows from $A\cap \Span( e_{i_1},\cdots, e_{i_{n-k+1}})\subset X$, and the properties of genus
that \[A\cap X_{n-k+1}=A\cap \Span( e_{i_1},\cdots, e_{i_{n-k+1}})\cap X=A\cap \Span( e_{i_1},\cdots, e_{i_{n-k+1}})\] is nonempty for any $A$ with $\gamma(A)\ge k$. Hence $\sup_{f\in A}I(f)\ge
\inf_{f\in X_{n+1-k}}I(f)=1$ and then $c_k=1$.

$\gamma\leq \kappa$: For a subset $W$ of $S_d$, let $I^1_W=\{f\in X: I(f)=1,f_i=0,\forall i\in S_d\setminus W\}$. Our purpose is to show $\gamma(I^1_{S_d})\leq \kappa$. We first prove that for any disjoint subsets $W_1$ and $W_2$, $I^1_{W_1\cup W_2}\subset I^1_{W_1}*I^1_{W_2}$. Here $*$ is the topological join\footnote{The topological join of two sets $A$ and $B$ in a linear space is usually defined to be $A*B:=\{ta+(1-t)b:\forall a\in A,~b\in B,~t\in[0,1]\}$.}. Given $ a\in I^1_{W_1\cup W_2}$, let
\[a^1_i=\begin{cases}
a_i,& \text{ if }\; i\in W_1\\
0,& \text{ if }\; i\not\in W_1
\end{cases}\;\;\text{and}\;\;a^2_i=\begin{cases}
a_i,& \text{ if }\; i\in W_2\\
0,& \text{ if }\; i\not\in W_2
\end{cases}\;\;\text{for}\;i\in V.\]
Since $W_1$ and $W_2$ are disjoint, one can easily verify that $ a= a^1+ a^2$ and $1=I( a)\leq I( a^1)+I( a^2)\leq \| a^1\|+\| a^2\|=\| a\|=1$. Hence, $I( a^1)=\| a^1\|$ and $I( a^2)=\| a^2\|$ hold. Taking $ f= a^1/\| a^1\|$ and $ g= a^2/\| a^2\|$, one has $I( f)=\| f\|=I( g)=\| g\|=1$, which implies that $ f\in I^1_{W_1}$ and $ g\in I^1_{W_2}$.
Therefore, we have $ a=\| a^1\| f+\| a^2\| g=t f+(1-t) g$, where $t=\| a^1\|$. Then we obtain \[I^1_{W_1\cup W_2}\subset\left\{t f+(1-t) g\,\left| \;  f\in I^1_{W_1},  g\in I^1_{W_2}, t\in [0,1]\right\}\right.=I^1_{W_1}*I^1_{W_2}.\]
Combining the subadditivity of Krasnoselski genus with respect to topological join (see  \cite{Zhang2017}), we have
\begin{equation}\label{eq:subadditivity}
\gamma(I^1_{W_1\cup W_2})\leq \gamma(I^1_{W_1}*I^1_{W_2})\leq \gamma(I^1_{W_1})+\gamma(I^1_{W_2}).
\end{equation}
According to the mathematical induction, we can easily deduce that
\begin{equation}\label{eq:gamma-inequality}
\gamma(I^1_{S_d})\leq \min\left\{\sum_{i=1}^l\gamma(I^1_{W_i}):W_1,\cdots,W_l\text{ form a partition of }S_d\right\},
\end{equation}
which provides a recurrence method to estimate a large complex by smaller one.

Note that for a clique $W$ of $S_d$, $I^1_W=\{f\in F:I(f)=1, f_jf_l\ge0,\forall j,l\in W; f_i=0,\forall i\not\in W\}$. We can easily construct an odd continuous function $F:I^1_W\to S^0=\{-1,1\}$ defined by $$F(f)=\begin{cases}
1,&\text{ if there exists } i\in W \text{ such that }f_i>0,\\
-1,&\text{ if there exists } i\in W \text{ such that }f_i< 0.
\end{cases}$$
This means that $\gamma(I^1_W)=1$.

Then \eqref{eq:gamma-inequality} implies that
\begin{align*}
\gamma(I^1_{S_d})&\leq \min\left\{\sum_{i=1}^l\gamma(I^1_{W_i}):\text{cliques }W_1,\cdots,W_l\text{ form a partition of }S_d\right\}
\\&=\min\left\{l=\sum_{i=1}^l1:\text{cliques }W_1,\cdots,W_l\text{ form a partition of }S_d\right\}=\kappa.
\end{align*}
\end{proof}

\begin{remark}
Theorem \ref{th:independent-clique} holds for both signless up 1-Laplacians
and signless down 1-Laplacians.
\end{remark}

\begin{remark}
The inequality provided in Theorem \ref{th:independent-clique} looks like the Lovasz Sandwich theorem:
$$\alpha\le\Theta\le\theta\le \kappa$$
where $\Theta$ is the Shannon capacity of a graph, and $\theta$ is the Lovasz number (or Lovasz theta function).
\end{remark}

\section{Constructions and their effect on the spectrum}
\label{section:construction}

This section follows the lines of Horak and Jost \cite{HorakJost2013}. We denote by $\spec(K)$ the set of all eigenvalues of $\Delta_1$. Consider the topological multiplicity, we use $\Spec(K)$ to denote the 
multiset of $\Delta_1$-eigenvalues.
Further, $A\mathop{\cup}\limits^\circ B$ is the multiset sum of two multisets $A$ and $B$.  
Our results are similar to 
 \cite{HorakJost2013}, but most of the proofs are different.  

\subsection{Wedges}
\begin{defn}
The combinatorial $k$-wedge sum of simplicial complexes $K_1$ and $K_2$, $K_1 \vee_{k} K_2$, is defined as the quotient of their disjoint union by
the identification $F_1 \sim F_2$, that is
$$
K_1 \vee_{k} K_2:=K_1 \sqcup K_2 /\{F_1\sim F_2\}
$$
where $F_1$ and $F_2$ are the $k$-dim simplicial faces in $K_1$ and $K_2$ respectively.
\end{defn}
This definition could be generalized to the $k$-wedge sum of arbitrary many simplicial complexes.
For example, the $1$-wedge of some tetrahedrons is shown as below.
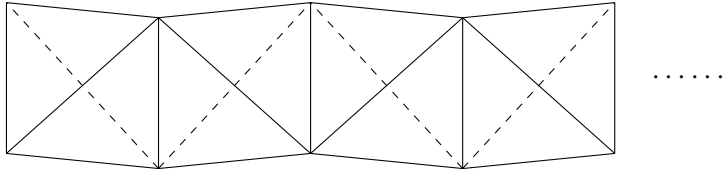
\begin{center}
\begin{tikzpicture}
\draw (0,0) to (0,2);
\draw (0,0) to (2,1.8);
\draw (0,0) to (2,-0.2);
\draw (2,1.8) to (2,-0.2);
\draw (2,1.8) to (0,2);
\draw[dashed] (2,-0.2) to (0,2);
\draw (4,0) to (4,2);
\draw (4,0) to (2,1.8);
\draw (4,0) to (2,-0.2);
\draw (2,1.8) to (4,2);
\draw[dashed] (2,-0.2) to (4,2);
\draw (4,0) to (6,1.8);
\draw (4,0) to (6,-0.2);
\draw (6,1.8) to (6,-0.2);
\draw (6,1.8) to (4,2);
\draw[dashed] (6,-0.2) to (4,2);
\draw (8,0) to (8,2);
\draw (8,0) to (6,1.8);
\draw (8,0) to (6,-0.2);
\draw (6,1.8) to (8,2);
\draw[dashed] (6,-0.2) to (8,2);
\node (A) at (9,1) {$\cdots\cdots$};
\end{tikzpicture}
\end{center}

\begin{theorem}\label{wedup}
$$\Spec(\Delta_{1,i}^{up}(K_1\vee_k K_2))=\Spec(\Delta_{1,i}^{up}(K_1))\mathop{\cup}\limits^\circ  \Spec(\Delta_{1,i}^{up}(K_2))$$
for all $i,k$ with $0\leq k< i$.
\end{theorem}

\begin{proof}
Since $K_1$ and $K_2$ are identified by a $k$-face, then by noticing that $k< i$, any $i$-face of $K_1$ is non-adjacent to $i$-faces of $K_2$ in $K_1\vee_k K_2$.
 Consequently, if $(\mu,f^1)$ is an eigenpair of $\Delta_{1,i}^{up}(K_1)$, then letting
$$f_j=\begin{cases}
f^1_j,& j\in S_i(K_1),\\
0,& j\in S_i(K_2),
\end{cases}$$
it is easy to check that $(\mu,f)$ is an eigenpair of $\Delta_{1,i}^{up}(K_1\vee_k K_2)$. The same property holds for $\Delta_{1,i}^{up}(K_2)$. Moreover, if  $(\mu,f^1)$ and $(\mu,f^2)$ are eigenpairs of $\Delta_{1,i}^{up}(K_1)$ and $\Delta_{1,i}^{up}(K_2)$, respectively, then is can be easily verified $(\mu,f)$ is an eigenpair of $\Delta_{1,i}^{up}(K_1\vee_k K_2)$, where $f$ is defined by
$$f_j=\begin{cases}
f^1_j,& j\in S_i(K_1),\\
f^2_j,& j\in S_i(K_2).
\end{cases}$$
So, we have proved that
$$\Spec(\Delta_{1,i}^{up}(K_1\vee_k K_2))\supset \Spec(\Delta_{1,i}^{up}(K_1))\mathop{\cup}\limits^\circ  \Spec(\Delta_{1,i}^{up}(K_2)).$$
For the converse, let $(\mu,f)$ be an eigenpair of $\Delta_{1,i}^{up}(K_1\vee_k K_2)$, and let
$f^1$ (resp. $f^2$) be the  restriction of $f$ on $S_i(K_1)$ (resp. $S_i(K_2)$). Since $f\ne 0$, at least one of $f^1$ and $f^2$ is not $0$. Suppose $f^1\ne 0$.  Then, there exist $z_{j_1\cdots j_{i+1}j}\in \Sgn(f_{j_1}+\cdots+f_{j_{i+1}}+f_j)$  such that
$$\sum_{j_1,\cdots,j_{i+1}}z_{j_1j_2\cdots j}\in \mu \deg^{up}_j \Sgn(f_j),$$
for any $j\in S_i(K_1\vee_k K_2)=S_i(K_1)\cup S_i(K_2)$. Therefore, for $j\in S_i(K_1)$, the above equation holds and thus $(\mu,f^1)$ is an eigenpair of $\Delta_{1,i}^{up}(K_1)$. If $f^2\ne 0$, then the same process deduces that $(\mu,f^2)$ is an eigenpair of $\Delta_{1,i}^{up}(K_2)$. Hence,
$$\Spec(\Delta_{1,i}^{up}(K_1\vee_k K_2))\subset \Spec(\Delta_{1,i}^{up}(K_1))\mathop{\cup}\limits^\circ  \Spec(\Delta_{1,i}^{up}(K_2)).$$
\end{proof}

\begin{remark}
This is a signless 1-Laplacian counterpart of Theorem 6.1 \cite{HorakJost2013}.
\end{remark}

Similar proof, we have the following
\begin{theorem}\label{weddown}
$$\Spec(\Delta_{1,i}^{down}(K_1\vee_k K_2))=\Spec(\Delta_{1,i}^{down}(K_1))\mathop{\cup}\limits^\circ  \Spec(\Delta_{1,i}^{down}(K_2))$$
for all $i,k$ with $i>k+1$.
\end{theorem}

\begin{theorem}
Let $K_1$ and $K_2$ be simplicial complexes, for which the spectrum of $\Delta_{1,i}^{up}(K_1)$ and $\Delta_{1,i}^{up}(K_2)$ both contain the eigenvalue $\mu$, and let $f^1$, $f^2$ be their corresponding eigenvectors. If an $i$-wedge $K=K_1\vee_i K_2$ is obtained by identifying $i$-faces $i_1$ and $i_2$, for which $f^1_{i_1}=f^2_{i_2}$, then the spectrum of
$\Delta_{1,i}^{up}(K)$ contains the eigenvalue $\mu$, too.
\end{theorem}
\begin{proof}
Note that we have identified $i_1$ with $i_2$ in $K$.  So, we can assume $S_i(K_1\vee_i K_2)=(S_i(K_1)\setminus\{i_1\}) \cup S_i(K_2)$. It is easy to see that
$$\deg^{up}_j(K)=\begin{cases}
\deg^{up}_j(K_1),& \text{ if }j\in S_i(K_1)\setminus\{i_1\},\\
\deg^{up}_j(K_2),& \text{ if }j\in S_i(K_2)\setminus\{i_2\},\\
\deg^{up}_{i_1}(K_1)+\deg^{up}_{i_2}(K_2),&\text{ if }j=i_2.
\end{cases}$$
Now we are going to prove that
$$f_j=\begin{cases}
f^1_j,& j\in S_i(K_1)\setminus\{i_1\},\\
f^2_j,& j\in S_i(K_2),
\end{cases}$$
is an eigenvector of $\Delta_{1,i}^{up}(K)$ corresponding to the eigenvalue $\mu$. In fact, since $(\mu,f^1)$ is an eigenpair of $\Delta_{1,i}^{up}(K_1)$,   there exist $z_{j_1\cdots j_{i+1}j}^1\in \Sgn(f_{j_1}^1+\cdots+f_{j+1}^1+f_j^1)$ and $z_j^1\in \Sgn(f_j^1)$ such that
$$\sum_{j_1,\cdots j_{i+1}}z_{j_1j_2\cdots j_{i+1}j}^1=\mu \deg^{up}_j(K_1) z_j^1,$$
for any $j\in S_i(K_1)\setminus \{i_1\}$, and
 $$\sum_{j_1,\cdots j_{i+1}}z_{j_1j_2\cdots i_1}^1=\mu \deg^{up}_{i_1}(K_1) z_{i_1}^1.$$
 Similarly, there exist $z_{j_1\cdots j_{i+1}j}^2\in \Sgn(f_{j_1}^2+\cdots+f_{j_{i+1}}^2+f_j^2)$ and $z_j^2\in \Sgn(f_j^2)$ such that
$$\sum_{j_1,\cdots j_{i+1}}z_{j_1j_2\cdots j}^2=\mu \deg^{up}_j(K_2) z_j^2,$$
for any $j\in S_i(K_2)$. Now we take
$$z_j=\begin{cases}
z^1_j,& \text{ if }j\in S_i(K_1)\setminus\{i_1\},\\
z^2_j,& \text{ if }j\in S_i(K_2)\setminus\{i_2\},\\
\frac{\deg^{up}_{i_1}(K_1)}{\deg^{up}_{i_1}(K_1)+\deg^{up}_{i_2}(K_2)}z_{i_1}^1+
\frac{\deg^{up}_{i_2}(K_2)}{\deg^{up}_{i_1}(K_1)+\deg^{up}_{i_2}(K_2)}z_{i_2}^2,&\text{ if }j=i_2,
\end{cases}$$
and
$$z_{j_1j_2\cdots j_{i+1} j}=\begin{cases}
z_{j_1j_2\cdots
j_{i+1}j}^1,& \text{ if }j\in S_i(K_1),\\
z_{j_1j_2\cdots j_{i+1}j}^2,& \text{ if }j\in S_i(K_2).
\end{cases}$$
Since $f^1_{i_1}=f^2_{i_2}$, then $\frac{\deg^{up}_{i_1}(K_1)}{\deg^{up}_{i_1}(K_1)+\deg^{up}_{i_2}(K_2)}z_{i_1}^1+
\frac{\deg^{up}_{i_2}(K_2)}{\deg^{up}_{i_1}(K_1)+\deg^{up}_{i_2}(K_2)}z_{i_2}^2\in \Sgn(f^1_{i_1})=\Sgn(f^2_{i_2})=\Sgn(f_{i_2})$, which means that $z_j$ is always well-defined. Finally, it can be easily verified that $$\sum_{j_1,\cdots ,j_{i+1}}z_{j_1j_2\cdots j_{i+1}j}=\mu \deg^{up}_j(K) z_j,$$
for any $j\in S_i(K)$, which completes the proof.
\end{proof}

\begin{remark}
This is a signless 1-Laplacian counterpart of Theorem 6.3 \cite{HorakJost2013}.
\end{remark}

\begin{theorem}\label{inter1}
Let $c_1,\cdots,c_m$ be the eigenvalues of $\Delta_{1,i}^{up}(K_1\cup K_2)$ and $c'_1,\cdots,c'_{m-1}$ the eigenvalues of $\Delta_{1,i}^{up}(K)$, where $K=K_1\vee_i K_2$, then $$c_j\leq c'_j$$
for every $0\leq j\leq m-1$.
\end{theorem}

\begin{proof}
Let $i_1$ and $i_2$ be $i$
-faces which are identified in an $i$-wedge sum $K$, which will be denoted by $i'$ in $K$.
Note that
$$I^{up}_{K_1\cup K_2}(f)=\sum_{j_1,\cdots,j_{i+2} \ne i_1,i_2}|f_{j_1}+\cdots+ f_{j_{i+2}}|+\sum_{j_2,\cdots,j_{i+2}}|f_{i_1}+f_{j_2}+\cdots+ f_{j_{i+2}}|+\sum_{j_2,\cdots,j_{i+2}}|f_{i_2}+f_{j_2}+\cdots+ f_{j_{i+2}}|$$
and
$$I^{up}_{K}(g)=\sum_{j_1,\cdots,j_{i+2} \ne i'}|g_{j_1}+\cdots+ g_{j_{i+2}}|+\sum_{j_2,\cdots,j_{i+2}}|g_{i'}+g_{j_2}+\cdots+ g_{j_{i+2}}|$$
where all $j_1,\ldots,j_{i+2}$ under the sum notation are $i$-faces of a $(i+1)$-simplex. It is easy to see that
$$I^{up}_{K_1\cup K_2}(f)-I^{up}_{K}(g)=\sum_{j_2,\cdots,j_{i+2}}|f_{i_1}+f_{j_2}+\cdots+ f_{j_{i+2}}|+\sum_{j_2,\cdots,j_{i+2}}|f_{i_2}+f_{j_2}+\cdots+ f_{j_{i+2}}|-\sum_{j_2,\cdots,j_{i+2}}|g_{i'}+f_{j_2}+\cdots +f_{j_{i+2}}|$$ whenever $g_j=f_j$ for $j\notin \{i',i_1,i_2\}$. And $I^{up}_{K_1\cup K_2}(f)-I^{up}_{K}(g)=0$ if we further assume that $g_{i'}=f_{i_1}=f_{i_2}$.

Let $X=\{f\in \mathbb{R}^m:\sum_{i=1}^m \deg^{up}_i(K_1\cup K_2)|f_i|=1\}$ and $Y=\{g\in \mathbb{R}^{m-1}:\sum_{i=1}^{m-1} \deg^{up}_i(K)|g_i|=1\}$, where $m=\# S_i(K_1)+\#S_i(K_2)$. Note that $$\deg^{up}_j(K)=\begin{cases}
\deg^{up}_i(K_1\cup K_2),& \text{ if }j\in S_i(K)\setminus\{i'\},\\
\deg^{up}_{i_1}(K_1)+\deg^{up}_{i_2}(K_2),&\text{ if }j=i'.
\end{cases}$$
Let $\hat{X}=X\cap \{f\in \mathbb{R}^m:f_{i_1}=f_{i_2}\}$ and let $\psi:\hat{X}\to Y$ be defined by
$$\psi(f)_j=\begin{cases}
f_j,& \text{ if }j\in S_i(K)\setminus\{i'\},\\
f_{i_2},&\text{ if }j=i'.
\end{cases}$$
Then $\psi$ is an odd homeomorphism from $\hat{X}$ to $Y$. And $I^{up}_{K}(\psi(f))=I^{up}_{K_1\cup K_2}(f)$, $\forall f\in \hat{X}$. Thus,
\begin{align*}
c_j&=\inf\limits_{\gamma(A)\ge j,A\subset X}\sup_{f\in A} I^{up}_{K_1\cup K_2}(f)
\\&\leq \inf\limits_{\gamma(A)\ge j,A\subset\hat{X}}\sup_{f\in A} I^{up}_{K_1\cup K_2}(f)
\\&= \inf\limits_{\gamma(A)\ge j,A\subset\hat{X}}\sup_{f\in A} I^{up}_{K}(\psi(f))
\\&= \inf\limits_{\gamma(B)\ge j,B\subset Y}\sup_{y\in B} I^{up}_{K}(g)=c'_j.
\end{align*}
The proof is completed.
\end{proof}

\subsection{Duplication of motifs}

Given a simplicial complex $K=(V,S)$  and a collection of simplicial faces $M$.
The closure $\text{Cl } M$ of $M$ is the smallest subcomplex of $K$ that contains each simplex in $M$ and is obtained by repeatedly adding to $M$ each face of every simplex in $M$. The star $\text{St } M$ of $M$ is the set of all simplices in $K$ that have a face in $M$. The link $\text{lk } M$ of $M$
is $\text{Cl St } M \setminus \text{St Cl } M$.
\begin{defn}[i-motif]
A subcomplex $M$ of a simplicial complex $K$
is an $i$-motif if:

$(1)$ $(\forall F_1,F_2 \in M)$, if $F_1,F_2 \subset F \in
K$, then $F\in M$

$(2)$ dim $\text{lk } M=i.$
\end{defn}
From the definition of link, the vertices in motif $M$ are different from that in $\text{lk } M.$
Let $l_0,...,l_m$ be vertices of $\text{lk } M$ and  $p_{0},...,p_{k}$ be the vertices of $M.$ Duplication of the
$i$-motif $M$ is defined as follows.

\begin{defn}[duplication of the $i$-motif $M$]
Let $M'$ be a simplicial complex on the vertices $p'_{0},...,p'_{k}$ and the map $h:p_i'\rightarrow p_i$ be a simplicial isomorphism between $M'$ and $M$. Let $K^{M}:=K\cup\{\{p'_{i_{0}},\cdots,p'_{i_{k}},l_{j_{1}},\cdots,l_{j_{l}}\}|
\{p_{i_{0}},\cdots,p_{i_{k}},l_{j_{1}},\cdots,l_{j_{l}} \} \in  K\}$. We call $K^{M}$ the duplication of $i$-motif of $M$.
\end{defn}
Note that $K=(K-\text{St } M)\vee_{i}(\text{Cl St } M)$, as a consequence of Theorem \ref{wedup} ,
we have the following
$$
\Spec(\Delta_{1,k}^{up}(K)=\Spec(\Delta_{1,k}^{up}(K-\text{St }M))\mathop{\cup}\limits^\circ  \Spec(\Delta_{1,k}^{up}(\text{Cl St }M))
\quad \text{for} \quad 0\leq i< k.
$$

The following proposition is proved by the similar methods in \cite{HorakJost2013}. For completeness, we give the proof.
\begin{pro}
For $i$-motif $M$, considering $\Delta^{up}_{1,i}(\text{Cl St }M)|_{\text{St }M}$ which is the restriction of
$\Delta^{up}_{1,i}(\text{Cl St }M)$ on ${\text{St }M},$ if $(\mu,h)$ is the eigenpair of $\Delta^{up}_{1,i}(\text{Cl St }M)|_{\text{St }M}$, then $(\mu,\varrho)$ is an eigenpair
of $\Delta^{up}_{1,i}(K^{M}),$ where
$$\varrho(F)=\begin{cases}
h(F),& \text{ if } F\in \text{St }M;\\
-h(F),&\text{ if } F \in \text{St }M';\\
0,&\text{ otherwise}.
\end{cases}$$

\end{pro}
\begin{proof}
According to the definition of $\text{St }M$, for any $F \in \text{St }M$, if $F \subset \overline F$, then $\overline F \in \text{St }M$, which implies that
$\Delta^{up}_{1,i}(\text{Cl St }M)$ and
$\Delta^{up}_{1,i}(K^{M})$ coincide on
$\text{St }M$ and for any $i$-face $F$ in $\text{St }M$, there is $\deg_F^{up}(\text{Cl St } M)=\deg_F^{up}(K^{M})$.
Let$$\varrho(F)=\begin{cases}
h(F),& \text{ if } F\in \text{St }M;\\
-h(F),&\text{ if } F \in \text{St }M';\\
0,&\text{ otherwise}.
\end{cases}$$
Then,
$$\Delta^{up}_{1,i}(K^{M})\varrho(F)=\begin{cases}
\Delta^{up}_{1,i}(\text{Cl St }M)|_{\text{St }M}h(F) \in \mu \deg^{up}_{F}\Sgn(\varrho (F) ),& \text{ for }\forall
F \in \text{St }M;\\
\Delta^{up}_{1,i}(\text{Cl St }M)|_{\text{St }M}(-h(F)) \in \mu \deg^{up}_{F}\Sgn(\varrho (F)),&\text{ for }\forall
F \in \text{St }M';\\
0\in\mu \deg^{up}_{F}Sgn(\varrho (F)),  &\text{ otherwise.}\\
\end{cases}$$
which confirms the claim.
\end{proof}
The same to \cite{HorakJost2013}, we have the following corollary.
\begin{cor}
If the spectrum of the simplicial complex $\text{Cl St }M$ contains the eigenvalue $\mu,$ with an eigenvector $h$ that is identically equal to zero on $\text{lk }M,$ then $\mu$ is also the eigenvalue of $K^{M}.$
\end{cor}

Using the same methods in the proof of Theorem \ref{inter1}, we have the following
\begin{theorem}\label{inter2}
Let $c_j$ be the eigenvalue of
$\Delta^{up}_{1,i}(\text{Cl St }M)|_{\text{St }M}$ and $c'_{j}$ be the eigenvalue of
$\Delta^{up}_{1,i}(\text{Cl St }M)$. Then
$$c'_{j}\leq c_j.$$
\end{theorem}
\begin{proof}
Let $X=\{f \in \text{St }M : \sum_{f_{i}\in \text{St }M}\deg^{up}_i |f_i|=1\}$ and $Y=\{g \in \text{Cl St }M :\sum_{g_{j} \in \text{Cl St }M}\deg^{up}_j |g_j| =1\}$ and $\widehat Y=Y\cap\{f: f_{m}=0,  \forall f_{m} \in \text{Cl St }M\setminus \text{St }M\}.$ It is obvious that
$F:\widehat Y \to X $ defined by
$$F(f)_j=\begin{cases}
f_j,& \text{ if }f_j\in \text{St }M,\\
0,&\text{ if }f_j \in \text{Cl St }M\setminus \text{St }M,
\end{cases}$$
is odd homemorphism.
Then,
\begin{align*}
c'_j&=\inf\limits_{\gamma(A)\ge j,A\subset Y}\sup_{f\in A} I^{up}_{\text{Cl St }M}(f)
\\&\leq \inf\limits_{\gamma(A)\ge j,A\subset
\widehat  {Y}}\sup_{f\in A} I^{up}_{\text{Cl St }M}(f)
\\&= \inf\limits_{\gamma(A)\ge j,A\subset\widehat{Y}}\sup_{f\in A} I^{up}_{\text{St }M}(F(f))
\\&= \inf\limits_{\gamma(B)\ge j,B\subset X}\sup_{g\in B} I^{up}_{\text{St }M}(g)\\
&=c_j.
\end{align*}

\end{proof}
\begin{remark}
Theorem \ref{inter1} and Theorem \ref{inter2} are counterparts of Theorem 6.4 and Theorem 6.12 in \cite{HorakJost2013} respectively. 
\end{remark}


\vspace{0.3cm}

{\bf Acknowledgements.} Xin Luo is supported by China Postdoctoral Science Foundation (No. 2019M660829).
Dong Zhang is supported by grant from the Project funded by China Postdoctoral Science Foundation (No. BX201700009).

%
%

\end{document}